\documentclass[reqno]{amsart}
\usepackage{amssymb,amsmath,amsthm}
\usepackage{color}
\usepackage{enumerate}
\usepackage{fullpage}
\usepackage{url}
\usepackage[unicode,pdfusetitle]{hyperref}
\usepackage[mathscr]{euscript}
\usepackage{graphicx}
\usepackage{mathtools}

\newtheorem{theorem}{Theorem}
\newtheorem*{theorem*}{Theorem}

\newtheorem*{definition*}{Definition}

\newtheorem*{question*}{Question}

\newtheorem*{conjecture*}{Conjecture}
\newtheorem{corollary}[theorem]{Corollary}
\newtheorem*{corollary*}{Corollary}

\newtheorem*{remark*}{Remark}
\newtheorem{proposition}[theorem]{Proposition}
\newtheorem*{proposition*}{Proposition}
\newtheorem{lemma}[theorem]{Lemma}
\newtheorem*{lemma*}{Lemma}

\newtheorem*{example*}{Example}
\newtheorem{fact}[theorem]{Fact}
\newtheorem*{fact*}{Fact}

\newtheorem*{claim*}{Claim}
\numberwithin{theorem}{section}
\numberwithin{equation}{section}

\theoremstyle{remark}
\newtheorem{claim}{Claim}

\DeclareMathOperator{\an}{an}
\DeclareMathOperator{\anexp}{an,exp}

\DeclareMathOperator{\dcl}{dcl}

\DeclareMathOperator{\Gr}{Graph}

\DeclareMathOperator{\intr}{int}

\DeclareMathOperator{\rk}{rk}

\DeclareMathOperator{\supp}{supp}
\DeclareMathOperator{\Th}{Th}
\DeclareMathOperator{\tp}{tp}

\newcommand{\N}{\mathbb{N}}

\newcommand{\Q}{\mathbb{Q}}
\newcommand{\R}{\mathbb{R}}
\newcommand{\T}{\mathbb{T}}
\newcommand{\Z}{\mathbb{Z}}

\newcommand{\cA}{\mathcal A}

\newcommand{\cL}{\mathcal L}

\newcommand{\cO}{\mathcal O}
\newcommand{\cP}{\mathcal P}
\newcommand{\cR}{\mathcal R}
\newcommand{\cS}{\mathcal S}
\newcommand{\cU}{\mathcal U}

\newcommand{\mfM}{\mathfrak{M}}

\newcommand{\fm}{\mathfrak{m}}
\newcommand{\fn}{\mathfrak{n}}

\newcommand{\bk}{\boldsymbol{k}}

\newcommand{\inv}{^{-1}}
\newcommand{\mycomment}[1]{}
\DefineNamedColor{named}{teal}{cmyk}{0.80,0.20,0.40,0.20}

\newcommand{\TO}{T_{\cO}}
\newcommand{\LO}{\cL_{\cO}}
\newcommand{\TM}{T_{\mfM}}
\newcommand{\LM}{\cL_{\mfM}}
\newcommand{\Ls}{\cL_s}
\newcommand{\Lgks}{\cL_{\Gamma,\bk,s}}

\newcommand{\dclL}{\dcl_\cL}
\newcommand{\rkL}{\rk_\cL}
\renewcommand{\preceq}{\preccurlyeq}

\renewcommand{\geq}{\geqslant}
\renewcommand{\leq}{\leqslant}
\renewcommand{\phi}{\varphi}
\renewcommand{\epsilon}{\varepsilon}
\renewcommand{\k}{{\boldsymbol{k}}}

\DeclareFontFamily{OMS}{smallo}{}
\DeclareFontShape{OMS}{smallo}{m}{n}{<->s*[.65]cmsy10}{}
\DeclareSymbolFont{smallo@m}{OMS}{smallo}{m}{n}
\DeclareMathSymbol{\smallo}{\mathord}{smallo@m}{79}

\author{Elliot Kaplan}
\author{Christoph Kesting}
\email{kaplae2@mcmaster.ca}
\email{kestingc@mcmaster.ca}
\title{A dichotomy for $T$-convex fields with a monomial group}
\address{Department of Mathematics and Statistics, McMaster University, Hamilton, Ontario, Canada}
\date{\today}
\begin{document}
\maketitle
\begin{abstract}
We prove a dichotomy for o-minimal fields $\cR$, expanded by a $T$-convex valuation ring (where $T$ is the theory of $\cR$) and a compatible monomial group. We show that if $T$ is power bounded, then this expansion of $\cR$ is model complete (assuming that $T$ is), it has a distal theory, and the definable sets are geometrically tame. On the other hand, if $\cR$ defines an exponential function, then the natural numbers are externally definable in our expansion, precluding any sort of model-theoretic tameness.
\end{abstract}
%------------------------------------------%
\section*{Introduction}
%------------------------------------------%
Let $T$ be a complete o-minimal theory extending the theory of real closed ordered fields in an appropriate language $\cL$. Let $\cR$ be a model of $T$. In~\cite{DL95}, van den Dries and Lewenberg studied the expansion of $\cR$ by a proper \textbf{$T$-convex subring}:\ a convex subset $\cO \subseteq \cR$ which is closed under all $\cL(\emptyset)$-definable continuous functions $f\colon \cR\to \cR$. They showed that all such expansions have the same elementary theory, denoted $\TO$, in the language $\LO \coloneqq \cL\cup\{\cO\}$, and that this theory eliminates quantifiers (after extending $\cL$ so that $T$ has quantifier elimination and a universal axiomatization). It follows that $\TO$ is weakly o-minimal.

In this paper, we study models $(\cR,\cO)\models \TO$ which are further expanded by a \textbf{monomial group}:\ a multiplicative subgroup $\mfM \subseteq \cR^>$ which is mapped isomorphically onto the value group $\cR^\times/\cO^\times$. Consider the following examples:
\begin{enumerate}
\item Let $\R_{\an}$ be the expansion of the real field $\R$ by functions which are real analytic on a neighborhood of the box $[-1,1]^n$, restricted to this box. Let $T_{\an} \coloneqq \Th(\R_{\an})$, in the language extending the language of ordered rings by these function symbols. This theory is o-minimal and model complete~\cite{Ga68,vdD86}. The field of Puiseux series $\R(\!(t^{1/\infty})\!)\coloneqq \bigcup_n\R(\!(t^{1/n})\!)$ admits an expansion to a model of $T_{\an}$, where each restricted analytic function on $\R$ is extended to the corresponding box in $\R(\!(t^{1/\infty})\!)$ via Taylor series expansion. The convex hull of $\R$ in $\R(\!(t^{1/\infty})\!)$, consisting of all series in which only non-negative exponents of $t$ appear, is $T_{\an}$-convex, and the subgroup $t^\Q = \{t^q:q \in \Q\}\subseteq \R(\!(t^{1/\infty})\!)^>$ is a monomial group.
\item Let $\R_{\anexp}$ further expand $\R_{\an}$ by the unrestricted exponential function and let $T_{\anexp}\coloneqq \Th(\R_{\anexp})$. Again, this theory is o-minimal and model complete~\cite{DM94}. The field $\T$ of logarithmic-exponential transseries admits an expansion to a model of $T_{\anexp}$; see~\cite[Corollary 2.8]{DMM97}. This field is essentially obtained from the field of Puiseux series over $\R$ by ``closing off'' under exponentials and logarithms. Once again, the convex hull of $\R$ is $T_{\anexp}$-convex, and the subgroup of \emph{transmonomials} (transseries obtained by exponentiating the purely infinite elements of $\T$) is a monomial group.
\end{enumerate}
In this paper, we show that $\R(\!(t^{1/\infty})\!)$, as a model of $T_{\an}$ expanded by predicates for the convex hull of $\R$ and the monomial group $t^\Q$, is still model complete. Additionally, this structure admits quantifier elimination in a slightly extended language and it is distal. While this structure is no longer weakly o-minimal (the subgroup $t^\Q$ is discrete), all definable unary subsets in $\R(\!(t^{1/\infty})\!)$ are the union of an open set and finitely many discrete sets. In contrast, the field $\T$, as a model of $T_{\anexp}$ expanded by a predicate for the group of transmonomials, was shown to be highly untame by Camacho~\cite[Theorem 4.11]{Ca18}. Explicitly, this structure defines the natural numbers $\N$ and is therefore at least as complex as Peano arithmetic.

As it turns out, the precise dividing line in our setting is whether the theory $T$ defines an exponential function. If $T$ does not define an exponential, then any definable function in any model of $T$ is eventually bounded by a power function by Miller's dichotomy~\cite{Mi96}. We can then use results of van den Dries~\cite{vdD97} and Tyne~\cite{Ty03} to prove a quantifier elimination result, thereby showing that the tameness properties enjoyed by the Puiseux series hold for any model of $T$ with a monomial group (assuming that the monomial group is compatible with the power functions). If $T$ does define an exponential, then $\N$ is externally definable in any model of $T$ with a monomial group that is compatible with the exponential.

We define exactly what we mean by a ``compatible'' monomial group in Section~\ref{sec:prelim}, where we also provide the necessary background on power boundedness and $T$-convex subrings. Our quantifier elimination result for power bounded $T$ is established in Section~\ref{sec:qe}, and we use this result to show that the value group and residue field are still stably embedded, even after adding a monomial group (stable embeddedness without the monomial group was shown by van den Dries~\cite{vdD97}). In Section~\ref{sec:d-min}, we use our quantifier elimination result to show that the unary definable sets in these expansions are unions of an open set and finitely many discrete sets, and in Section~\ref{sec:distal} we show that the theory of these expansions is distal. We turn our attention to exponential $T$ in Section~\ref{sec:exponential}, where we show that the natural numbers are externally definable in any model of $T$ expanded by a monomial group. 

\subsection*{Acknowledgements}
Research for this paper was conducted in part at the Fields Institute for Research in Mathematical Sciences. The first author was supported by the National Science Foundation under Award No.\ 2103240.

%------------------------------------------%
\section{Preliminaries}\label{sec:prelim}
%------------------------------------------%
\subsection*{Notation and conventions}
We always use $k$, $m$, and $n$ to denote elements of $\N =\{0,1,2,\ldots\}$. If $S$ is a totally ordered set, then by a \textbf{cut} in $S$, we mean a downward closed subset of $S$. If $A$ is a cut in $S$ and $y$ is an element in an ordered set extending $S$, then we say that \textbf{$y$ realizes the cut $A$} if $A < y < S\setminus A$. For an arbitrary subset $A \subseteq S$, we let $A^\downarrow$ denote the downward closure of $A$, so $A^\downarrow = \{y \in S: y \leq a \text{ for some }a \in A\}$ is a cut in $S$. Given an ordered abelian group $\Gamma$, we let $\Gamma^>$ denote the set $\{\gamma \in \Gamma:\gamma>0\}$. Given a ring $R$, we let $R^\times$ denote the multiplicative group of units in $R$.

%------------------------------------------%
\subsection*{O-minimality}
Throughout, $\cL$ is a language extending the language $\{0,1,<,+,-,\cdot\}$ of ordered rings, and $T$ is a complete o-minimal theory extending the theory of real closed ordered fields. It is well-known that $T$ has definable Skolem functions, and consequently, we may arrange that $T$ has quantifier elimination and a universal axiomatization just by extending $\cL$ by function symbols for all $\cL(\emptyset)$-definable functions. Let $\cR \models T$. Then for $A \subseteq \cR$, we have that $\dclL(A)$ (the $\cL$-definable closure of $A$), is an elementary substructure of $\cR$ (again, as a consequence of definable Skolem functions). It follows that $T$ has a \textbf{prime model} $\cP$, which admits a unique embedding into any other model of $T$ with image $\dclL(\emptyset)$. Given an elementary extension $\cS$ of $\cR$ and a subset $A \subseteq \cS$, we denote by $\cR\langle A\rangle$ the intermediate extension $\dclL(\cR \cup A)\subseteq \cS$. When $A$ is just a singleton $\{a\}$, we write $\cR\langle a \rangle$ for this extension. The definable closure $\dclL$ is a pregeometry, and we define $\rkL(\cS|\cR)$ to be the cardinality of a $\dclL$-basis for $\cS$ over $\cR$ (that is, a subset $A\subseteq \cS$ which is $\dclL$-independent over $\cR$ such that $\cS = \cR\langle A \rangle$). If $\rkL(\cS|\cR) = 1$, then $\cS$ is said to be a \textbf{simple $T$-extension} of $\cR$, and $\cS = \cR\langle a \rangle$ for some $a \in \cS\setminus \cR$.

A \textbf{power function} is an $\cL(\cR)$-definable endomorphism of the ordered multiplicative group $\cR^>$. Each power function $f$ can be thought of as the function $x\mapsto x^\lambda$, where $\lambda\coloneqq f'(1) \in \cR$. The collection $\Lambda$ of all such $\lambda$ is a subfield of $\cR$, called the \textbf{field of exponents of $\cR$}. By Miller's dichotomy~\cite{Mi96}, either $\cR$ is \textbf{power bounded} (every definable function is eventually bounded by a power function) or $\cR$ defines an \textbf{exponential function} (an ordered group isomorphism $\exp\colon\cR\to \cR^>$ which is equal to its own derivative). If $\cR$ is power bounded, then every power function is $\cL(\emptyset)$-definable, and every other model of $T$ is also power bounded with the same field of exponents as $\cR$ (we just say \textbf{$T$ is power bounded}, and we call $\Lambda$ the \textbf{field of exponents of $T$}); see~\cite[Propositions 4.2 and 4.3]{Mi96}.
If $\cR$ defines an exponential function $\exp$, then $\exp$ is $\cL(\emptyset)$-definable; see the beginning of Section 2 in~\cite{Mi96}.

%------------------------------------------%
\subsection*{\texorpdfstring{$T$}{T}-convex subrings}
As stated in the introduction, a \textbf{$T$-convex subring} of $\cR$ is a convex subset of $\cR$ which is closed under all $\cL(\emptyset)$-definable continuous functions $f\colon \cR\to \cR$. Let $\LO\coloneqq \cL\cup\{\cO\}$ and let $\TO$ be the $\LO$-theory which extends $T$ by axioms stating that $\cO$ is a proper $T$-convex subring. Let $(\cR,\cO)\models\TO$. Then $\cO$ is a convex subring of $\cR$ (hence a valuation ring), and $(\cR,\cO)$ is a convexly valued ordered field. As each element of the prime model $\cP$ is $\cL(\emptyset)$-definable, we always have $\cP \subseteq \cO$. We let $\Gamma\coloneqq \cR^\times/\cO^\times$ denote the \textbf{value group} of $(\cR,\cO)$, written additively, and we let $v\colon \cR^\times \to \Gamma$ denote the surjective valuation map. If $T$ is power bounded with field of exponents $\Lambda$, then $\Gamma$ has the structure of an ordered $\Lambda$-vector space, where $\lambda\cdot v(a)\coloneqq v(a^\lambda)$ for $\lambda \in \Lambda$ and $a \in \cR^\times$. In fact, $\Gamma$ is stably embedded as an ordered $\Lambda$-vector space~\cite[Theorem 4.4]{vdD97}. We write $\smallo$ for the unique maximal ideal of $\cO$, and we let $\k\coloneqq \cO/\smallo$ denote the \textbf{residue field} of $(\cR,\cO)$. We let $\pi\colon \cO \to \k$ be the corresponding residue map; then $\pi$ is order-preserving as $\cO$ is convex. The residue field $\k$, considered with its induced structure, is a model of $T$; see~\cite[Remark 2.16]{DL95} and~\cite[Remark 2.3]{Yi17}. Moreover, $\k$ is stably embedded as a model of $T$~\cite[Corollary 1.13]{vdD97}. Accordingly, we always construe $\k$ as an $\cL$-structure.
We sometimes include the subscript $\cR$ on $\cO$, $\bk$, and $\Gamma$ when confusion may otherwise arise. Let $(\cS, \cO_\cS)$ be a $\TO$-extension of $(\cR,\cO_{\cR})$, so $(\cS,\cO_{\cS})$ is a model of $\TO$ which is also an $\LO$-extension of $(\cR,\cO_{\cR})$. If $\cS$ is a simple $T$-extension of $\cR$, then we say that $(\cS,\cO_{\cS})$ is a \textbf{simple $\TO$-extension} of $(\cR,\cO_{\cR})$.

\begin{fact}[\cite{DL95}, Remark 3.8]\label{two ring extensions} Let $(\cR,\cO_\cR)\models \TO$ and let $\cS$ be a simple $T$-extension of $\cR$. There are at most two $T$-convex valuation rings $\cO_1$ and $\cO_2$ of $\cS$ which make $\cS$ a $T_{\cO}$-extension of $\cR$:
$$\cO_1\coloneqq\{y \in \cS:|y|<u \text{ for some }u \in \cO_\cR\}, \quad \cO_2\coloneqq\{y \in \cS:|y|<d \text{ for all }d \in \cR\text{ with }d>\cO_\cR\}.$$
If the cut $\cO_\cR^\downarrow$ in $\cR$ is realized by $b \in \cS$, then $b$ belongs to $\cO_2$ but not $\cO_1$, so $\cO_1\subsetneq \cO_2$. If no element in $\cS$ realizes the cut $\cO_\cR^\downarrow$, then $\cO_1=\cO_2$.
\end{fact}

\begin{corollary}\label{two ring extensions2}
Let $(\cR,\cO_\cR)\models \TO$ and let $(\cS,\cO_\cS)$ be a simple $\TO$-extension of $(\cR,\cO_\cR)$. If $\Gamma_\cS = \Gamma_\cR$, then 
\[
\cO_{\cS}=\{y \in \cS:|y|<d \text{ for all }d \in \cR\text{ with }d>\cO_\cR\}
\]
If $\bk_{\cS} = \bk_{\cR}$, then 
\[
\cO_\cS =\{y \in \cS:|y|<u \text{ for some }u \in \cO_\cR\}.
\]
\end{corollary}
\begin{proof}
This is clear if no element in $\cS$ realizes the cut $\cO_\cR^\downarrow$. Suppose that $b \in \cS$ realizes this cut. If 
\[
\cO_{\cS}=\{y \in \cS:|y|<d \text{ for all }d \in \cR\text{ with }d>\cO_\cR\},
\]
then $\pi(b) \in \bk_{\cS} \setminus \bk_{\cR}$. Likewise, if 
\[
\cO_{\cS} =\{y \in \cS:|y|<u \text{ for some }u \in \cO_\cR\},
\]
then $v(b) \in \Gamma_\cS \setminus \Gamma_\cR$. The corollary follows by Fact~\ref{two ring extensions}.
\end{proof}

The theory $\TO$ is tame, regardless of whether $T$ is power bounded. However, when $T$ is power bounded, van den Dries showed that we have an analog of the Abhyankar-Zariski inequality, called the \emph{Wilkie inequality}:

\begin{fact}[The Wilkie inequality {\cite[Section 5]{vdD97}}]
Suppose that $T$ is power bounded with field of exponents $\Lambda$. Let $(\cR,\cO_\cR)\models \TO$, let $(\cS,\cO_\cS)$ be a $\TO$-extension of $(\cR,\cO_\cR)$, and suppose that $\rkL(\cS|\cR)$ is finite. Then 
\[
\rkL(\cS|\cR)\ \geq\ \rkL(\k_{\cS}|\k_\cR)+\dim_\Lambda(\Gamma_{\cS}/\Gamma_\cR)
\]
\end{fact}

We will often use the following consequence of this inequality:

\begin{corollary}\label{wilkie corollary}
Suppose that $T$ is power bounded, let $(\cR,\cO_\cR)\models \TO$, and let $(\cS,\cO_\cS)$ be a simple $\TO$-extension of $(\cR,\cO_\cR)$. Then either $\k_{\cS}=\k_\cR$ or $\Gamma_{\cS}=\Gamma_\cR$.
\end{corollary}

Of course, it may be the case that for an extension $(\cR,\cO_\cR)\preceq (\cS,\cO_\cS)\models \TO$, we have that both $\k_{\cS}=\k_\cR$ and $\Gamma_{\cS}=\Gamma_\cR$. In this case, $(\cS,\cO_{\cS})$ is said to be an \textbf{immediate extension} of $(\cR,\cO_\cR)$.

Let $(a_\rho)$ be a well-indexed sequence of elements of $\cR$, so $\rho$ ranges ordinals less than $\lambda$ for some limit ordinal $\lambda$. We say that $(a_\rho)$ is \textbf{pseudocauchy} if there is an index $\rho_0$ such that
\[
v(a_\tau- a_\sigma)\ >\ v(a_\sigma - a_\rho)
\]
for all $\rho_0 < \rho < \sigma < \tau < \lambda$. 
An element $a$ in some $\TO$ extension of $\cR$ is called a \textbf{pseudolimit of $(a_\rho)$} if there is an index $\rho_0$ such that
\[
v(a- a_\sigma)\ >\ v(a - a_\rho)
\]
for all $\rho_0 < \rho < \sigma<\lambda$. If $(\cS,\cO_\cS)$ is an immediate extension of $(\cR,\cO_\cR)$, then for any $a \in \cS\setminus \cR$, there is a pseudocauchy sequence $(a_\rho)$ in $\cR$ with pseudolimit $a$ and with no pseudolimits in $\cR$; see~\cite{Ka42}.

%------------------------------------------%
\subsection*{Monomial groups}\label{subsec:monomials}
Let $(\cR,\cO) \models \TO$. A \textbf{section of $v$} is a group homomorphism $s \colon \Gamma\to \cR^\times$ such that $v\circ s$ is the identity on $\Gamma$.
A \textbf{monomial group} for $(\cR,\cO)$ is a multiplicative subgroup $\mfM \subseteq \cR^\times$ such that $v\colon \mfM\to \Gamma$ is a group isomorphism. If $s$ is a section of $v$, then $s(\Gamma)$ is a monomial group, and if $\mfM$ is a monomial group, then $(v|_{\mfM})\inv$ is a section of $v$. Note that any monomial group is necessarily a subgroup of $\cR^>$, as $\cR$ is real closed and $\Gamma$ is divisible. 

In this paper, we will restrict our attention to monomial groups that are compatible with the o-minimal structure on $\cR$. These monomial groups should respect the power functions in the case that $T$ is power bounded and the exponential function when $T$ is not power bounded. More precisely, we say that a monomial group $\mfM \subseteq \cR^>$ is \textbf{$T$-compatible} if either
\begin{enumerate}
\item $T$ is power bounded and $\mfM$ is closed under all power functions, or
\item $T$ defines an exponential function $\exp$ and $\mfM^\succ \coloneqq \{\fm \in \mfM: \fm>1\}$ is closed under $\exp$.
\end{enumerate}
We say that a section $s$ of $v$ is $T$-compatible if the corresponding monomial group $s(\Gamma)$ is $T$-compatible. If $s\colon \Gamma\to \cR^>$ is $T$-compatible and $T$ is power bounded with field of exponents $\Lambda$, then $s$ is an ordered $\Lambda$-vector space embedding.

\begin{lemma}
Any model $(\cR,\cO) \models \TO$ admits a $T$-compatible monomial group $\mfM$.
\end{lemma}
\begin{proof}
If $T$ defines an exponential function, then the existence of a $T$-compatible monomial group follows Ressayre's dyadic representation of real closed exponential fields~\cite[Theorem 4]{Re93}. Though Ressayre builds a compatible monomial group with respect to the archimedean valuation and the exponential function $2^x$, his methods adapt to the construction of a compatible monomial group for any exponential and any $T$-convex valuation ring. Suppose $T$ is power bounded with field of exponents $\Lambda$. Then $\cR^>$ is an ordered $\Lambda$-vector space and $(\cO^\times)^>\coloneqq \{u \in \cR^>:v(u) = 0\}$ is a $\Lambda$-subspace of $\cR^>$, so we may take a $\Lambda$-subspace $\mfM\subseteq \cR^>$ which is a $\Lambda$-vector space complement to $(\cO^\times)^>$. Then every $a \in \cR^\times$ can be uniquely represented as a product of some $\fm \in \mfM$ and some $u \in \cO^\times$, so $\mfM$ is a $T$-compatible monomial group for $(\cR,\cO)$.
\end{proof}

Let $\LM\coloneqq \LO\cup\{\mfM\}$, and let $\TM$ be the $\LM$-theory which extends $\TO$ by axioms stating that $\mfM$ is a $T$-compatible monomial group.

%------------------------------------------%
\section{Quantifier elimination for power bounded \texorpdfstring{$T$}{T}}\label{sec:qe}
%------------------------------------------%
In this section, we assume that $T$ is power bounded with field of exponents $\Lambda$. We will show that $\TM$ is model complete. This model completeness is a by-product of a quantifier elimination proof in an expanded language: Let $\Lgks$ be the three-sorted language with sorts for $\cR$ and the residue field $\bk_\cR$, both in the language $\cL$, and a sort for the value group $\Gamma_\cR$ in the language of ordered $\Lambda$-vector spaces. We include a function symbol $\pi\colon \cR \to \bk_\cR$ for the residue map (defined to be zero off of the valuation ring), a function symbol $v\colon \cR^\times \to \Gamma_\cR$ for the valuation map, and a function symbol $s\colon \Gamma_\cR \to \cR^>$ for the $T$-compatible section corresponding to $\mfM$. We do not include relation symbols for $\cO$ and $\mfM$ in the sort for $\cR$, but these predicates are $\Lgks(\emptyset)$-definable: the monomial group $\mfM$ is defined by $\{y \in \cR^>: s(v(y)) = y\}$ and the valuation ring $\cO$ is defined by $\{y \in \cR: v(y)\geq 0\}$. Any model $\cR = (\cR,\cO,\mfM)\models \TM$ admits a unique expansion to an $\Lgks$-structure $(\cR,\Gamma_\cR,\bk_\cR)$, and if $\cR\subseteq \cS$ are models of $\TM$, then the expansion $(\cR,\Gamma_\cR,\bk_\cR)$ is an $\Lgks$-substructure of the expansion $(\cS,\Gamma_\cS,\bk_\cS)$. It is therefore harmless to refer to the three-sorted $\Lgks$-structure $(\cR,\Gamma_\cR,\bk_\cR)$ also as a model of $\TM$.

\begin{theorem}\label{thm:qe}
Suppose $T$ has quantifier elimination and a universal axiomatization. Then $\TM$ has quantifier elimination in the language $\Lgks$.
\begin{proof}
Let $(\cR,\Gamma_\cR,\bk_\cR)\models \TM$ and $(\cS,\Gamma_\cS,\bk_\cS)\models \TM$ be $|\cR|^+ $-saturated. Let $(\cA,\Gamma_\cA,\bk_\cA)$ be a common $\Lgks$-substructure. To show that $\TM$ has quantifier elimination, we need to show that the inclusion $(\cA,\Gamma_\cA,\bk_\cA)\subseteq (\cS,\Gamma_\cS,\bk_\cS)$ extends to an $\Lgks$-embedding $(\cR,\Gamma_\cR,\bk_\cR)\to (\cS,\Gamma_\cS,\bk_\cS)$; see~\cite[Corollary B.11.9]{ADH17}. As $T$ already has quantifier elimination and a universal axiomatization, both $\cA$ and $\bk_\cA$ are models of $T$ and $\Gamma_\cA$ is a $\Lambda$-vector space. Note that the valuation $v\colon \cA^\times \to \Gamma_\cA$ is surjective, as it has a section, but that the residue map $\pi\colon \cA\to \bk_\cA$ need not be surjective.
\begin{claim}\label{claim1}
By quantifier elimination for $T$ and saturation, we can extend 
the inclusion $\bk_\cA \subseteq \bk_\cS$ to an $\cL(\bk_\cA)$-embedding $\bk_\cR\to \bk_\cS$.
\end{claim}
\begin{claim}\label{claim2}
We can extend the inclusion $\Gamma_\cA \subseteq \Gamma_\cS$ to an embedding $\Gamma_\cR\to \Gamma_S$. Suppose $\alpha \in \Gamma_\cR \setminus \Gamma_\cA$. Using saturation, take $\beta \in \Gamma_{\cS}$, realizing the same cut as $\alpha$ over $\Gamma_{\cA}$. Put $\fm \coloneqq s(\alpha)$ and $\fn \coloneqq s(\beta)$, and note that $\fn$ realizes the same cut over $\cA$ as $\fm$, so we get an $\cL$-embedding $f\colon \cA\langle \fm\rangle\to \cS$ which sends $\fm$ to $\fn$. 
By Corollary~\ref{wilkie corollary}, we have $\pi(\cA\langle \fm\rangle)= \pi(\cA\langle \fn\rangle)= \pi(\cA)$. It follows from Corollary~\ref{two ring extensions2} that for $y \in \cA\langle \fm \rangle$, we have
\[
v(y)\geq 0\ \Longleftrightarrow\ |y|<u\text{ for some }u \in \cA\text{ with }v(u)\geq 0\ \Longleftrightarrow\ v(f(y))\geq 0
\]
so $f$ is even an $\LO$-embedding. In order to extend $f$ to an $\Lgks$-embedding
\[
(\cA\langle \fm\rangle,\Gamma_{\cA}\oplus \Lambda \alpha,\bk_\cR)\to (\cS,\Gamma_{\cS},\bk_{\cS}),
\]
it remains to note that for $\gamma + \lambda\alpha \in \Gamma_{\cA}\oplus\Lambda\alpha$, we have
\[
f(s(\gamma+ \lambda\alpha))\ =\ f(s(\gamma)\fm^\lambda)\ =\ s(\gamma)\fn^\lambda\ =\ s(\gamma+ \lambda\beta).
\]
\end{claim}
\begin{claim}\label{claim3}
We can extend the inclusion $\cA \subseteq \cS$ such that $\pi : \cA \to \bk_\cA$ is surjective. Suppose $\bar{a}\in \bk_\cA\setminus \pi(\cA)$. Let $a \in \cA$ and $b \in \cS$ be lifts of $\bar{a}$. Note that $a$ and $b$ both realize the cut $$\{y \in \cA : y < \cO_\cA\} \cup \{y\in \cO_\cA: \pi(y)<\bar{a} \},$$ 
so there is an $\cL$-embedding $f\colon \cA\langle a\rangle\to \cS$ sending $a$ to $b$. 
By Corollary~\ref{wilkie corollary}, we have that $\Gamma_\cA=\Gamma_{\cA\langle a \rangle}= \Gamma_{\cA \langle b \rangle}$. Then by Corollary~\ref{two ring extensions2}, we have for $y \in \cA\langle a \rangle$ that 
\[
v(y)\geq 0\ \Longleftrightarrow\ |y|<d\text{ for all }d \in \cA^>\text{ with }v(d)< 0\ \Longleftrightarrow\ v(f(y))\geq 0.
\]
Thus $f$ is even an $\LO$-embedding, so it induces an $\Lgks$-embedding $(\cA\langle a \rangle,\Gamma_\cA, \bk_\cR)\to (\cS,\Gamma_\cS, \bk_\cS)$.
\end{claim}
Now we extend the inclusion $(\cA,\Gamma_\cA,\bk_\cA)\subseteq (\cS,\Gamma_\cS,\bk_\cS)$ to an $\Lgks$-embedding $(\cR,\Gamma_\cR,\bk_\cR)\to (\cS,\Gamma_\cS,\bk_\cS)$. 
By the previous claims, we can arrange that $\Gamma_\cA=\Gamma_\cR$ and $\pi(\cA)=\bk_\cA = \bk_\cR$, so $\cR$ is an immediate extension of $\cA$. Let $a \in \cR \setminus \cA$, and take a pseudocauchy sequence $(a_\rho)$ in $\cA$ with pseudolimit $a$ and with no pseudolimits in $\cA$. Let $b \in \cS$ be a pseudolimit of $(a_\rho)$. Then by~\cite[Corollary 2.11]{Ka23}, there is a unique $\LO (\cA)$-embedding $f\colon \cA\langle a \rangle \to \cS$ sending $a$ to $b$. This $f$ induces an $\Lgks$-embedding $(\cA\langle a \rangle, \Gamma_\cR, \bk_\cR) \to (\cS, \Gamma_\cS, \bk_\cS)$.\qedhere
\end{proof}
\end{theorem}

\begin{corollary}\label{cor:mc}
The theory $\TM$ is complete. If $T$ is model complete, then $\TM$ is also model complete in the language $\LM$.
\end{corollary}
\begin{proof}
Let $\cP$ be the prime model of $T$. Then $(\cP,\{0\},\cP)$ admits an $\Lgks$-embedding into any model of $\TM$ since $\k \models T$, so $\TM$ is complete; see~\cite[Corollary B.11.7]{ADH17}. For model completeness in the language $\LM$, let $\cR$ and $\cS$ be models of $\TM$ and assume that $\cS$ is $|\cR|^+$-saturated. Let $\cA\models \TM$ be a common $\LM$-substructure of $\cR$ and $\cS$. By a variant of Robinson's model completeness test~\cite[Corollary B.10.4]{ADH17}, it is enough to show that the inclusion $\cA\subseteq\cS$ extends to an embedding $\cR\to \cS$. As $T$ is model complete, $\TO$ is as well by~\cite[Corollary 3.13]{DL95}, so $\cA$ is an elementary $\LO$-substructure of both $\cR$ and $\cS$. %It follows that for $\fm \in \mfM_\cR$, if $v(\fm) \in v(\cA^\times)$, then $\fm \in \mfM_\cA$. 

Extending our language $\cL$ by function symbols for $\cL(\emptyset)$-definable functions, we arrange that $T$ has quantifier elimination and a universal axiomatization. Augmenting by additional sorts for the value group and residue field, we view $\cR$ and $\cS$ as $\Lgks$-structures $(\cR,\Gamma_\cR,\bk_{\cR})$ and $(\cS,\Gamma_\cS,\bk_{\cS})$, where $s\colon \Gamma_\cR\to \cR^>$ is the section corresponding to the monomial group $\mfM$, and similarly for $\cS$. 
Given $\gamma \in v(\cA^\times)$, we have $v(s(\gamma)) = \gamma \in v(\cA^\times)$, so $s(\gamma)$ belongs to $\mfM_\cA$. Thus, $(\cA,v(\cA^\times),\pi(\cA))$ is a common $\Lgks$-structure of $(\cR,\Gamma_\cR,\bk_{\cR})$ and $(\cS,\Gamma_\cS,\bk_{\cS})$. Theorem~\ref{thm:qe} gives an $\Lgks$-embedding $(\cR,\Gamma_\cR,\bk_{\cR})\to (\cS,\Gamma_\cS,\bk_{\cS})$ over $(\cA,\Gamma_\cA,\bk_{\cA})$, which restricts to an $\LM$-embedding $\cR\to \cS$ over $\cA$.
\end{proof}

If $\cL$ is finite (that is, if $\cL$ has only finitely many non-logical symbols) and $T$ is decidable, then $T_\mfM$ is certainly effectively axiomatizable. By completeness of $T_\mfM$ and~\cite[Corollary B.6.9]{ADH17}, we deduce the following:

\begin{corollary}
If $\cL$ is finite and $T$ is decidable, $\TM$ is decidable.
\end{corollary}

We can also draw the usual corollaries about stable embeddedness and orthogonality. 
\begin{corollary}\label{cor:stembed}
The value group $\Gamma$ is purely stably embedded as an ordered $\Lambda$-vector space and orthogonal to the residue field, which is purely stably embedded as a model of $T$.
\begin{proof}
By extending $\cL$ by function symbols for all $\cL(\emptyset)$-definable functions, we may assume that $T$ has quantifier elimination and a universal axiomatization.
Let $(\cR,\Gamma_\cR,\bk_\cR)\models \TM$ and let $(\cA,\Gamma_\cA,\bk_\cA)$ be an $\Lgks$-substructure of $(\cR,\Gamma_\cR,\bk_\cR)$. Let $\gamma,\gamma'\in \Gamma_\cR^m$ with $\tp(\gamma/\Gamma_\cA)=\tp(\gamma'/\Gamma_\cA)$ in the language of ordered $\Lambda$-vector spaces, and let $r,r'\in \bk_\cR^n$ with $\tp_\cL(r/\bk_\cA)=\tp_\cL(r'/\bk_\cA)$. By a standard compactness argument, it suffices to show that $(\gamma,r)$ has the same $\Lgks$-type as $(\gamma',r')$ over $(\cA,\Gamma_\cA,\bk_\cA)$; see~\cite{CDH05}. As $r$ and $r'$ have the same type over $\bk_\cA$, we find an $\cL$-isomorphism $\bk_\cA \langle r \rangle\to \bk_\cA \langle r' \rangle$ mapping $r$ to $r'$. Using Claim~\ref{claim1} of the quantifier elimination proof, this extends to an $\Lgks$-isomorphism
$$(\cA , \Gamma_\cA, \bk_{\cA}\langle r \rangle) \to ( \cA , \Gamma_\cA, \bk_{\cA}\langle r' \rangle).$$
Let $\fm,\fm' \in \cR^m$ be the tuples $(s(\gamma_1),\ldots,s(\gamma_m))$ and $(s(\gamma_1'),\ldots,s(\gamma_m'))$, respectively. Then $v(\cA\langle \fm\rangle^\times) = \Gamma_{\cA}\oplus \Lambda \gamma_1\oplus \cdots\oplus \Lambda\gamma_m$, and similarly for $v(\cA\langle \fm'\rangle^\times)$. By Claim~\ref{claim2} of the quantifier elimination proof, we get an $\Lgks$-isomorphism
$$(\cA \langle \fm \rangle , v(\cA\langle \fm\rangle^\times), \bk_{\cA}\langle r \rangle) \to ( \cA \langle \fm' \rangle , v(\cA\langle \fm'\rangle^\times), \bk_{\cA}\langle r' \rangle).$$
This isomorphism is elementary by our quantifier elimination, giving us stable embeddedness and orthogonality. Purity follows directly by taking our substructure to be $(\cP,\{0\},\cP)$, where $\cP$ is the prime model of $T$. 
 \end{proof}
\end{corollary}

We note that as a consequence of Corollary~\ref{cor:stembed}, the monomial group $\mfM$ is itself purely stably embedded as an ordered $\Lambda$-vector space.

%------------------------------------------%
\section{Definable sets}\label{sec:d-min}
%------------------------------------------%
In this section and the next, we will establish some consequences of our quantifier elimination result. For this, it will be more convenient to work in a one-sorted language in which we still have quantifier elimination. Let $\Ls$ extend $\cL$ by a unary function symbol $s$. We interpret a model $\cR = (\cR,\cO,\mfM) \models \TM$ as an $\Ls$ model by putting 
\[
s(a) = \fm \in \mfM\ :\Longleftrightarrow\ v(a) = v(\fm)
\]
for $a \in \cR^\times$ and by setting $s(0)\coloneqq 0$. Note that 
\[
\mfM = \{y \in \cR^>: s(y) = y\}\text{ and }\cO=\{y \in \cR: s(y)\leq s(1)\}
\]
are both $\Ls$-definable.% and that, conversely, $s$ is $\LM$-definable.% (for nonzero $y$, $s(y)$ is the unique element of $\mfM$ such that $y/\fm$ and $\fm/y$ both belong to $\cO$).

\begin{corollary}\label{cor:1sortedqe}
Suppose $T$ has quantifier elimination and a universal axiomatization. Then $\TM$ has quantifier elimination in the language $\Ls$.
\end{corollary}
\begin{proof}
Let $\cR$ and $\cS$ be models of $\TM$, and assume that $\cS$ is $|\cR|^+$-saturated. Let $\cA$ be a common $\Ls$-substructure of $\cR$ and $\cS$. As in the proof of Corollary~\ref{cor:mc}, we augment by additional sorts for the value group and residue field to get that $(\cA,v(\cA^\times),\pi(\cA))$ is a common $\Lgks$-substructure of $(\cR,\Gamma_\cR,\bk_{\cR})$ and $(\cS,\Gamma_\cS,\bk_{\cS})$, where the sections from the value group sort to the field sort are defined using the map $s$. Theorem~\ref{thm:qe} gives an $\Lgks$-embedding $(\cR,\Gamma_\cR,\bk_{\cR})\to (\cS,\Gamma_\cS,\bk_{\cS})$ over $(\cA,v(\cA^\times),\pi(\cA))$, which restricts to an $\Ls$-embedding $\cR\to \cS$ over $\cA$.
\end{proof}

For the rest of this section, let $\cR\models \TM$ and let $A \subseteq \cR$ be a set of parameters.

\begin{lemma}\label{lem:termprep}
Let $\tau$ be a unary $\Ls(A)$-term. Then there is $m \in \N$, an $(m+1)$-ary $\cL(A)$-definable function $f\colon \cR^{m+1}\to \cR$, and an $\Ls(A)$-definable set $B \subseteq \cR^{m+1}$ such that 
\begin{enumerate}
\item $B_x$ (the fiber of $B$ over $x$) is open for each $x \in \cR^m$, and $B_x \cap B_y = \emptyset$ for $x \neq y \in \cR^m$.
\item $\cR \setminus \pi^*(B)$ is a finite union of $\Ls(A)$-definable discrete sets, where $\pi^*(B) = \bigcup_{x \in R^m}B_x$. 
\item For each $x \in \cR^m$, we have $\tau(t) = f(x,t)$ for all $t \in B_x$.
\end{enumerate}
\end{lemma}
\begin{proof}
We proceed by induction on the complexity of terms. If $\tau$ is a variable or a constant symbol, then we take $m = 0$, $f(t) = \tau(t)$, and $B = \cR$ (here $\cR^0$ is the one-point space).

Suppose that the lemma holds for all terms of lower complexity than $\tau$. We first consider the case that $\tau = \sigma(\tau_1,\ldots,\tau_n)$ for $\Ls$-terms $\tau_1,\ldots,\tau_n$ and an $\cL$-term $\sigma$. For each $i =1,\ldots,n$, take $m_i \in \N$, an $\cL(A)$-definable function $f_i\colon \cR^{m_i+1}\to \cR$, and an $\Ls(A)$-definable set $B_i \subseteq \cR^{m_i+1}$ satisfying the conditions in the lemma for $\tau_i$. Let $m \coloneqq m_1+\ldots+m_n$ and define $B \subseteq \cR^{m+1}$ and $f\colon \cR^{m+1}\to \cR$ as follows: for $x = (x_1,\ldots,x_n) \in \cR^{m_1} \times\cdots\times \cR^{m_n}$, put
\[
B_x\ \coloneqq\ B_{x_1}\cap \cdots \cap B_{x_n},\qquad f(x,t) = \sigma\big(f_1(x_1,t),\ldots,f_n(x_n,t)\big).
\]
Then $m$, $B$, and $f$ satisfy the conditions in the lemma for $\tau$.

Finally, suppose that $\tau = s(\sigma)$ for some $\Ls$-term $\sigma$. Take $m \in \N$ and an $\cL(A)$-definable function $g\colon \cR^{m+1}\to \cR$, and an $\Ls(A)$-definable set $C \subseteq \cR^{m+1}$ satisfying the conditions of the lemma for $\sigma$. As $g$ is $\cL(A)$-definable, there are $\cL(A)$-definable functions $g_1,\ldots,g_k:\cR^m\to \cR$ such that $t \mapsto g(x,t)\colon \cR\to \cR$ is continuous on $\cR \setminus \{g_1(x),\ldots,g_k(x)\}$ for all $x \in \cR^m$. We let $C^*\coloneqq C \setminus \bigcup_{i = 1}^k \Gr(g_i)$. Note that then $t \mapsto g(x,t)$ is continuous on the open set $C^*_x = C_x \setminus \{g_1(x),\ldots g_k(x)\}$ for each $x$. Now we define $B \subseteq \cR^{m+2}$ as follows: for $x \in \cR^m$ and $y \in \cR$, set
\[
B_{x,y} \coloneqq \left\{
\begin{array}{ll}
\big\{t \in C^*_x: s(g(x,t)) = y\big\} & \mbox{if $y \in \mfM$}\\
 \intr\!\big(\{t \in C^*_x :g(x,t) = 0\}\big) & \mbox{if $y = 0$}\\
\emptyset & \mbox{otherwise}.
\end{array}\right.
\]
Then each $B_{x,y}$ is open, since $s^{-1}(y)$ is open for $y \in \mfM$ and $t\mapsto g(x,t)$ is continuous on $C^*_x$ for each $x$. Clearly, the sets $B_{x,y}$ are pairwise disjoint. Let $f\colon \cR^{m+2}\to \cR$ be given by $f(x,y,t) = y$. Then for $(x,y) \in \cR^{m+1}$ and $t \in B_{x,y}$, we have
\[
\tau(t)\ =\ s(\sigma(t))\ =\ s(g(x,t))\ =\ y\ =\ f(x,y,t).
\]
It remains to show that $\cR \setminus \bigcup_{x,y}B_{x,y}$ is a finite union of $\Ls(A)$-definable discrete sets. By assumption, $\cR \setminus \bigcup_xC_x$ is a finite union of $\Ls(A)$-definable discrete sets, so it suffices to show that the $\Ls(A)$-definable set $\bigcup_x \big(C_x \setminus \bigcup_yB_{x,y}\big)$ is discrete. Since each $C_x$ is open, it is enough to show that $C_x \setminus \bigcup_yB_{x,y}$ is finite for each $x$, and this holds since $C_x \setminus \bigcup_yB_{x,y}$ is contained in union of $\{g_1(x),\ldots,g_k(x)\}$ and the boundary of the set $\{t \in \cR:g(x,t) = 0\}$.
\end{proof}

Using this lemma, we can describe the subsets of $\cR$.
\begin{theorem}\label{thm:dmin}
Every $\Ls(A)$-definable subset of $\cR$ is the union of an $\Ls(A)$-definable open set and finitely many $\Ls(A)$-definable discrete sets.
\end{theorem}
\begin{proof}
Let $D \subseteq \cR$ be $\Ls(A)$-definable. By removing the interior of $D$ (which is open and $\Ls(A)$-definable), we may assume that $D$ has empty interior. We will show that $D$ is a finite union of $\Ls(A)$-definable discrete sets. By quantifier elimination, we may assume that $D$ is of the form
\[
D\ =\ \{ t \in \cR: \tau_0(t) = 0,\tau_1(t)<0,\ldots,\tau_n(t)<0\}
\]
for $\Ls$-terms $\tau_0,\ldots,\tau_n$. For each $i \leq n$, take $m_i \in \N$, an $\cL(A)$-definable function $f_i\colon \cR^{m_i+1}\to \cR$, and an $\Ls(A)$-definable set $B_i \subseteq \cR^{m_i+1}$ as in Lemma~\ref{lem:termprep}. Let $m \coloneqq m_0+\ldots+m_n$ and for $x = (x_0,\ldots,x_n) \in \cR^{m_0} \times\cdots\times \cR^{m_n}$, set $B_x \coloneqq B_{x_0}\cap \cdots \cap B_{x_n}$. Then each $B_x$ is open, $\cR \setminus \bigcup_{x \in \cR^m}B_x$ is a finite union of $\Ls(A)$-definable discrete sets, and for each $x$, we have
\[
D \cap B_x\ =\ \{t \in B_x: f_0(x_0,t) = 0,f_1(x_1,t) < 0,\ldots,f_n(x_n,t)<0\}.
\]
As each $f_i$ is $\cL(A)$-definable and $B_x$ is open, we see that $D\cap B_x$ is finite (otherwise, $D\cap B_x$ has interior). Thus, $D \cap \bigcup_{x \in \cR^m}B_x$ is $\Ls(A)$-definable and discrete. As $\cR \setminus \bigcup_{x \in \cR^m}B_x$ is a finite union of $\Ls(A)$-definable discrete sets, we conclude that $D$ is a finite union of $\Ls(A)$-definable discrete sets. 
\end{proof}

Structures in which all unary definable sets are a union of an open set and finitely many discrete sets are sometimes called \emph{d-minimal}, though d-minimal structures are often additionally assumed to be definably complete. Of course, our structure is not definably complete, as the valuation ring is bounded but has no supremum in $\cR$. In~\cite[Section 9]{Fo11}, Fornasiero gives a more relaxed definition of a d-minimal structure which does not include definable completeness. We do not know whether $\cR$ is d-minimal in this sense.

The first example of a d-minimal structure is the expansion of the real field $\R$ by a predicate for the multiplicative subgroup $2^\Z$~\cite{vdD85}. D-minimality was further developed by Miller and others; see~\cite{Mi05}. The theory $\TM$ is quite similar to the theory of $(\R,2^\Z)$: our proof of Theorem~\ref{thm:dmin} has the same structure as the proof in~\cite{Mi05} that $(\R,2^\Z)$ is d-minimal. Moreover, the proof we give in the next section that $\TM$ is distal follows the proof in~\cite{HN17} that $(\R,2^\Z)$ has a distal theory. It is also worth mentioning here previous work of Scowcroft~\cite{Sc88}, who proved a version of Theorem~\ref{thm:dmin} for the field of $p$-adic numbers with the canonical section of the $p$-adic valuation (which is interdefinable with a predicate for the monomial group~$p^\Z$).

%------------------------------------------%
\section{Distality}\label{sec:distal}
%------------------------------------------%
Distality is a model-theoretic dividing line introduced by Pierre Simon in~\cite{Si13}, that aims to capture order-like behavior within dependent (or NIP) theories. A theory is \textbf{distal} if for every indiscernible $(a_i)_{i\in I}$ and any parameter set $A$ such that
\begin{enumerate}[(a)]
\item $I = I_1+(c)+I_2$ where $I_1,I_2$ are infinite without endpoints, and
\item $(a_i)_{i \in I_1+I_2}$ is $A$-indiscernible,
\end{enumerate}
then the entire sequence $(a_i)_{i \in I}$ is $A$-indiscernible as well. 
As $\TO$ is weakly o-minimal, it is dp-minimal by~\cite[Theorem 4.1]{DGL11} and hence distal by~\cite[Lemma 2.10]{Si13}.

\begin{theorem}\label{thm:dist}
$\TM$ is distal.
\end{theorem}
\begin{proof}
Let $(\cU,\cO_\cU,\mfM_\cU) \models \TM$ be a monster model. As in the previous section, we assume that $T$ has quantifier elimination and a universal axiomatization, and we work in the language $\Ls$, so $s(\cU^\times) = \mfM_\cU$. We will use the Hieronymi-Nell criterion for distality~\cite[Theorem 2.1]{HN17}, applied to our theory $T$ with additional function symbol $s$. We need to verify the following:
\begin{enumerate}
\item The theory $\TM$ has quantifier elimination in the language $\Ls$.
\item For every $\Ls$-substructure $\cR \subseteq \cU$ and every $c \in \cU^m$, there is a tuple $d \in s(\cR\langle c\rangle)^n$ for some $n$ such that 
$s(\cR\langle c\rangle) \subseteq \langle s(\cR),d \rangle$.
\item Suppose that $k'\leq k$ and $g,h$ are $\cL$-terms of arities $k+m$ and $k'+n$ respectively, $b_1\in \cU^m$, and $b_2\in\mfM_{\cU}^n$. If $(a_i)_{i\in I}$ is an $\Ls(\emptyset)$-indiscernible sequence from $\mfM_{\cU}^{k'}\times \cU^{k- k'}$ such that
\begin{enumerate}
\item $I = I_1+(c)+I_2$, where $I_1$ and $I_2$ are infinite without endpoints, and $(a_i)_{i\in I_1+I_2}$ is $\Ls(b_1b_2)$-indiscernible, and
\item $s(g(a_i,b_1)) = h(a_i,b_2)$ for every $i\in I_1+I_2$, 
\end{enumerate}
then $s(g(a_c,b_1)) = h(a_c,b_2)$.
\end{enumerate}
We have already verified (1) in Corollary~\ref{cor:1sortedqe} above. For (2), let $\cR$ be an $\Ls$-substructure of $\cU$ and let $c \in \cU^m$. Then $s(\cR\langle c\rangle)$ is a finitely generated multiplicative $\Lambda$-vector space over $s(\cR)$ by the Wilkie inequality. Take generators $\fm_1,\ldots,\fm_n \in s(\cR\langle c\rangle)$. Then
\[
s(\cR\langle c\rangle) \subseteq \langle s(\cR),\fm_1,\ldots,\fm_n \rangle.
\]
Finally, let $f,g,(a_i),b_1,b_2$ be as in (3), so $s(g(a_i,b_1)) = h(a_i,b_2)$ for every $i\in I_1+I_2$. We may as well assume that $g(a_i,b_1)$ and $h(a_i,b_2)$ are nonzero for these $i$ (otherwise, $s(g(a_c,b_1)) = h(a_c,b_2) = 0$ as well, since $T$ is distal).
We first claim that $h(a_i,b_2) \in \mfM_\cU$ for all $i\in I$. Fix $i \in I_1+I_2$, so $h(a_i,b_2)= h(a_{i,1},\ldots,a_{i,k'},b_2) \in \mfM_\cU$. Let $\cR$ be the $\cL$-substructure of $\cU$ generated by $(a_{i,1},\ldots,a_{i,k'},b_2)$. Since $(a_{i,1},\ldots,a_{i,k'},b_2) \in \mfM_{\cU}^{k'+n}$, the Wilkie inequality tells us that $s(\cR)$ is the multiplicative $\Lambda$-vector space generated by $(a_{i,1},\ldots,a_{i,k'},b_2)$, so in particular
\begin{equation}\label{eq:hequal}
h(a_i,b_2)\ =\ a_{i,1}^{\lambda_1}\cdots a_{i,k'}^{\lambda_{k'}}b_2^\lambda
\end{equation}
for some $\lambda_1,\ldots,\lambda_{k'} \in \Lambda$ and some tuple $\lambda \in \Lambda^n$. Since $T$ is distal, the equality (\ref{eq:hequal}) holds for all $i \in I$. Thus $h(a_c,b_2)$ is a product of $\Lambda$-powers of elements in $\mfM_{\cU}$, so $h(a_c,b_2) \in \mfM_{\cU}$. Therefore, in order to show that $s(g(a_c,b_1)) = h(a_c,b_2)$, it is enough to show that $v(g(a_c,b_1))= v(h(a_c,b_2))$. This holds since $v(g(a_i,b_1)) =v( h(a_i,b_2))$ for all $i \in I_1+I_2$ and since $\TO$ is distal.
\end{proof}

%Furthermore, about dividing lines we can show:
\begin{proposition}
$\TM$ is dependent (has NIP). However, $\TM$ is not strongly dependent.
\end{proposition}
\begin{proof}
All distal theories are dependent. To see that $\TM$ is not strongly dependent, let $(\cU,\cO_\cU,\mfM_\cU)\models \TM$ be sufficiently saturated, and note that for each $\epsilon \in \cU^{>0}$, the set $\mfM_{\cU} \cap (0,\epsilon)$ is definable, discrete, and infinite. By~\cite[Theorem 2.11]{DG17}, $\TM$ is not strong.
\end{proof}

%------------------------------------------%
\section{Exponential \texorpdfstring{$T$}{T}}\label{sec:exponential}
%------------------------------------------%
In this section, we assume that $T$ defines an exponential function $\exp$.
Let $\cR = (\cR,\cO,\mfM)\models \TM$. Recall our assumption that $\mfM^{\succ}$ is closed under $\exp$. We denote the compositional inverse of $\exp$ by $\log$. 

\begin{lemma}\label{lem:decomp}
The additive group of $\cR$ admits the direct sum decomposition $\cR = \cO \oplus \log(\mfM)$.
\end{lemma}
\begin{proof}
We claim that $\exp(\cO) = (\cO^\times)^>$. For one inclusion, let $a \in \cO$, and note that both $\exp a$ and $(\exp a)\inv = \exp(-a)$ belong to $\cO^>$ by $T$-convexity and $\cL(\emptyset)$-definability of $\exp$, so $a \in (\cO^\times)^>$. For the other inclusion, let $u \in (\cO^\times)^>$. If $u\geq 1$, then $\log u \in \cO$, since $\cO$ is convex and $0\leq \log u <u$. If $u < 1$, then $u\inv > 1$ and so $\log(u\inv) \in \cO$ by the previous case. Therefore $\log u = -\log(u\inv) \in \cO$ as well. The decomposition $\cR = \cO \oplus \log(\mfM)$ follows from our claim and the fact that $\cR^>$ is an internal (multiplicative) direct sum of $(\cO^\times)^>$ and $\mfM$.
\end{proof}

We now follow Camacho's strategy for showing that Hahn fields with a predicate for the subring of purely infinite elements are undecidable~\cite[Section 4.2]{Ca18}. Let $a \in \cR$ and $\fm \in \mfM$. By Lemma~\ref{lem:decomp}, there is a unique $b \in \fm \log(\mfM)$ with $a-b \in \fm\cO$. We define $a|_{\fm}$ to be this element $b$, so $(a,\fm)\mapsto a|_{\fm}$ is an $\LM(\emptyset)$-definable function. We also define
\[
\supp(a)\ \coloneqq\ \{\fm \in \mfM: s(a- a|_{\fm}) =\fm\},
\]
so $\supp(a)$ is an $\LM(a)$-definable subset of $\mfM$.
The element $a|_{\fm}$ functions as a sort of the ``truncation of $a$ at $\fm$,'' and $\supp(a)$ serves as an analog of the support. Indeed, viewing the field of transseries $\T$ as a model of $T_{\an,\exp}$, the element $a|_{\fm}$ is exactly the truncation of an element $a\in \T$ at a transmonomial $\fm$, and the set $\supp(a)$ is exactly the support of $a$. If $\fm$ is an infinitesimal transmonomial in $\T$, then the support of $(1-\fm)^{-1}$ is the set $\{\fm^n:n\in \N\}$. Thus, $\N = \{\log\fn/\log \fm:\fn \in \supp(1-\fm)^{-1}\}$ is definable in $\T$. We will show that something similar holds in our model $\cR$.

\begin{proposition}\label{prop:suppofinv}
Let $\fm \in \mfM$ with $\fm < 1$. Then $\fm^n \in \supp((1-\fm)\inv)$ for all $n \in \N$, and if $\fn \in \supp((1-\fm)\inv)$, then either $\fn = \fm^n$ for some $n$ or $\fn < \fm^n$ for all $n$.
\end{proposition}
\begin{proof}
Since $\fm\in \mfM$ is less than $1$, it belongs to the maximal ideal $\smallo$, so $(1-\fm)$ and its inverse $(1-\fm)\inv$ both belong to $\cO^\times$. We have $(1-\fm)\inv|_1 = 0$, so $1 \in \supp((1-\fm)\inv)$ and if $\fn \in \supp((1-\fm)\inv)$, then $\fn \leq 1$. Let us now fix $n \in \N$ and $\fn \in \mfM$ with $\fm^{n+1} \leq \fn < \fm^n$. We will show that $\fn \in \supp((1-\fm)\inv)$ if and only if $\fn = \fm^{n+1}$. We have
\[
1+ \fm+\cdots+\fm^n \ =\ \fn(\fn\inv+ \fn\inv\fm+\cdots+\fn\inv\fm^n).
\]
Since $\fn\inv,\ldots, \fn\inv\fm^n$ are all in $\mfM^\succ \subseteq \log(\mfM)$, their sum is in $\log(\mfM)$ as well, so $1+ \fm+\cdots+\fm^n$ belongs to $\fn\log(\mfM)$. We have
\[
\frac{1}{1-\fm}- 1- \fm-\cdots-\fm^n \ =\ \frac{\fm^{n+1}}{1-\fm}\ \in\ \fn\cO,
\]
so $(1-\fm)\inv|_\fn = 1+ \fm+\cdots+\fm^n$. Since $s((1-\fm)\inv-(1-\fm)\inv|_\fn) = \fm^{n+1}$, we conclude that $\fn \in \supp((1-\fm)\inv)$ if and only if $\fn = \fm^{n+1}$.
\end{proof}

\begin{corollary}\label{cor:Nset}
There is a definable set $A\subseteq \cR$ with $\N \subseteq A$ such that if $a\in A\setminus \N$, then $a>\N$. Consequently, $\N$ is externally definable in any model of $\TM$, and if $T$ has an archimedean model, then $\N$ is definable in some model of $\TM$.
\end{corollary}
\begin{proof}
Fix $\fm \in \mfM$ with $\fm<1$ and let $A$ be the definable set 
\[
A\ \coloneqq\ \big\{a \in \cR: \exp(a\log \fm) \in \supp\big((1-\fm)\inv\big)\big\}.
\]
Then $a \in A$ if and only if $a \in \N$ or $a > \N$, by Proposition~\ref{prop:suppofinv}. It follows that $\N$ is externally definable in $\cR$ since it is the intersection of $A$ with a convex subset of $\cR$. Suppose now that $T$ has an archimedean model. Then there is a model of $\cS\models \TM$ where $\cO_{\cS} = \{a\in \cS: |a|<n\text{ for some }n \in \N\}$. Defining $A$ in this model as above, we have $\N = A\cap \cO_{\cS}$.
\end{proof}

The definability of such a set $A$ as above precludes the possibility of any ``tame'' model-theoretic behavior such as distality, dependence, or even NTP$_2$. By taking parameters from the initial segment $\N\subseteq A$, we can transfer model-theoretic combinatorial properties from $\N$ to $\cR$. As an illustration, we will show that the theory of $\cR$ has the antichain tree property (ATP), as described in~\cite{AKL23}. A theory has ATP if there is a formula $\phi(x,y)$ and a tuple of parameters $(a_\eta)_{\eta \in 2^{<\omega}}$ such that $\{\phi(x,a_\eta): \eta \in I\}$ is consistent if and only $I \subseteq 2^{<\omega}$ is an antichain. Among theories with the strict order property (SOP), the antichain tree property implies all the other ``non-tame'' combinatorial properties studied in model theory so far, such as TP$_2$~\cite{AK20} (of course, the theory of $\cR$ has SOP as well).

\begin{corollary}
Any completion of $\TM$ has ATP.
\end{corollary}
\begin{proof}
We argue as in~\cite[Example 4.31]{AKL23}. Assume that $\cR$ is sufficiently saturated, let $A$ be the definable set from Corollary~\ref{cor:Nset}, and let $\phi(x,y)$ be the formula which states that $x \in A\setminus\{1\}$ and that $x\cdot z = y$ for some $z \in A$. We need to find a tuple of parameters $(a_\eta)_{\eta \in 2^{<\omega}}$ such that $\{\phi(x,a_\eta): \eta \in I\}$ is consistent if and only $I \subseteq 2^{<\omega}$ is an antichain. By saturation, it is enough to find for each $n$, a tuple of parameters $(a_\eta)_{\eta \in 2^{<n}}$ such that $\{\phi(x,a_\eta): \eta \in I\}$ is consistent if and only $I \subseteq 2^{<n}$ is an antichain. Fix $n$, let $I_1,\ldots,I_m$ enumerate the antichains in $2^{<n}$, and let $p_1,\ldots,p_m$ enumerate the first $m$ prime numbers. For each $\eta \in 2^{<n}$, let $a_\eta$ be the product of the primes $p_j$ for which $\eta \in I_j$.
Now let $I\subseteq 2^{<n}$ be an arbitrary subset. If $I$ is an antichain, then $I = I_j$ for some $j \in \{1,\ldots,m\}$ and so $\{\phi(x,a_\eta): \eta \in I\}$ is consistent, as witnessed by $p_j$. Conversely, suppose that $\{\phi(x,a_\eta): \eta \in I\}$ is consistent, as witnessed by some $b \in A\setminus\{1\}$. Then $b$ is less than each $a_\eta$, so $b$ belongs to $\N$ since each $a_\eta$ is in $\N$. Take a prime factor $p$ of $b$. Then $p$ is also a prime factor of each $a_\eta$, so $p = p_j$ for some $j\in \{1,\ldots,m\}$. But then $I \subseteq I_j$, so $I$ is an antichain.
\end{proof}

\end{document}